  \documentclass[amscd,amssymb,verbatim,12pt]{amsart}
\usepackage{epsfig}
\usepackage{graphicx}
\newif\ifAMS
\IfFileExists{amssymb.sty}
  {\AMStrue\usepackage{amssymb}}
  {\usepackage{latexsym}}
  \usepackage{enumerate}
\theoremstyle{plain}

\newtheorem{Thm}{Theorem}[section]
\newtheorem{Cor}[Thm]{Corollary}
\newtheorem{Lem}[Thm]{Lemma}

\theoremstyle{definition}
\newtheorem{Def}{Definition}

\theoremstyle{remark}
\newtheorem{Rem}{Remark}

\DeclareMathOperator{\length}{length}

\newcommand{\interior}{^{ \kern-5pt ^\circ}}

\newcommand {\Z}{{\mathbb Z}}

\begin{document}
\title
{Periodic geodesics in singular spaces}

\author
{Panos Papasoglu, Eric Swenson }

\subjclass{}

\address  [Panos Papasoglu]
{Mathematical Institute, University of Oxford, Andrew Wiles Building, Woodstock Rd, Oxford OX2 6GG, U.K.  }
\email {} \email [Panos Papasoglu]{papazoglou@maths.ox.ac.uk}

\address
[Eric Swenson] {Mathematics Department, Brigham Young University,
Provo UT 84602}
\email [Eric Swenson]{eric@math.byu.edu}

%\subjclass{53C20, 49Q20, 05C40}%%MSC class 

\begin{abstract} 
We extend the classical result of Lyusternik and Fet on the existence of closed geodesics to singular spaces. We show that if $X$ is a compact geodesic metric space satisfying the CAT($\kappa $) condition for some fixed $\kappa >0$ and $\pi _n(X)\ne 0$ for some $n>0$ then $X$ has a periodic geodesic. This condition is satisfied
for example by locally CAT($\kappa $) manifolds.
Our result applies more generally to compact locally uniquely geodesic spaces.
\end{abstract}
%\thanks{This work was supported by the Engineering and
%Physical Sciences Research Council [grant number EP/I01893X/1]}
\maketitle
\section{Introduction}

The question of the existence of periodic geodesics in closed Riemannian manifolds was first
considered by Poincar\'e in \cite {Po}. Birkhoff \cite{Bi} proved existence of periodic geodesics
for the sphere $S^n$ and Fet-Lyusternik extended this to every closed Riemannian manifold \cite{FL}.
For a review on the subject we refer to section 10.4 of \cite {Be}. 

Crucial to the existence result is Birkhoff's shortening process. We refer to \cite {Kl} appendix A
for an exposition of the Fet-Lyusternik result. Another more modern exposition of the topic is given in ch. 5 of \cite{CM}.

It is natural to ask whether the existence result for periodic geodesics applies to wider classes of spaces
and whether this really requires a Riemannian metric. All existing proofs to our knowledge rely on analytic methods where a Riemannian metric appears to be necessary.

Gruber in \cite{Gr} showed that generically boundaries of convex bodies in $\mathbb R ^3$ do not have any periodic geodesics. This shows that one does need to impose some restrictions on the space.

We show in this paper that one can prove the existence result under a quite weak assumption.
One needs only assume that $X$ is a geodesic metric space homeomorphic to a close manifold
with the property that there is an $\epsilon >0$ so that any two points at distance $\leq \epsilon $ are joined by a unique shortest path. In fact the hypothesis that $X$ is homeomorphic to a closed manifold
is not needed either, it suffices to assume that $\pi _n(X)\ne 0$ for some $n>0$.

Our proof of this more general result is geometric and in some ways simpler than previous proofs.

It applies in particular to closed manifolds with curvature bounded above in the sense of Alexandrov, i.e. manifolds that
satisfy locally the $CAT(\kappa )$ condition. These spaces are extensively studied, we refer to the classic \cite{AZ} and to the more recent text \cite{BH} for the foundations
and to \cite{AKP}, ch II, sec 9 for an up to date exposition.
The second author wishes to thank the Max Planck Institute for Mathematics for their support while working on this result. We thank S. Sabourau for many useful comments on a first draft of our paper.

\section {Preliminaries}

Let $X$ be a metric space. We recall that the length of a continuous path $\gamma :[0,\ell]\to X$
is defined as
$$\length (\gamma )=\sup \{\sum _{i=1}^nd(\gamma (t_i),\gamma (t_{i+1})):\ \ 0=t_1<...<t_{n+1}=\ell , n\in \mathbb N\}.$$

\begin{Def}  Let $X$ be a metric space. A continuous path $\gamma :[0,\ell]\to X$
is called a {\it shortest path} if $$d(\gamma (0),\gamma (\ell))=\length (\gamma ).$$ 

We say that the path $\gamma $ is a {\it geodesic} if there is an $\epsilon $ 
such that $$d(\gamma (t),\gamma (s))=\length (\gamma |_{[t,s]})$$
if $d(\gamma (t),\gamma (s))<\epsilon $.

We define similarly what it mean to be a geodesic for
paths $\gamma :S^1\to X$ and we call such paths {\it periodic geodesics}.

We say that $X$ is a {\it geodesic metric space} if any two points in $X$ can be joined by a shortest path.

We say that the path  $\gamma :[0,\ell]\to X$ is a {\it piecewise shortest path} if there is a partition of $[0,\ell]$
and $\gamma $ is a shortest path on each closed interval of the partition.

Let $X$ be a compact geodesic metric space. We say that $X$ is {\it  locally uniquely geodesic} if there is an $\epsilon >0 $ such that any two points
$x,y\in X$ with $d(x,y)\leq \epsilon $ can be joined by a unique shortest path. We will also use $\epsilon$-locally uniquely geodesic when we wish to fix the $\epsilon$ in question.

\end{Def}

It is convenient to parametrize geodesics by arc-length or proportionally to arc length. We will do this from now on, so when we state that $\gamma :[0,1]\to X$ is a piecewise shortest path or a geodesic it will be implicit that
$\gamma $ is parametrized proportionally to arc length, unless we specify a different parametrisation.

\begin{Rem} By \cite[II Proposition 1.4 (1)]{BH} $CAT(\kappa)$ spaces are locally uniquely geodesic. So Riemannian manifolds and Riemannian polyhedra satisfying the $CAT(\kappa )$ condition are locally uniquely geodesic.
\end{Rem}

%\begin{Rem} If $X$ is {\it $\epsilon $-locally uniquely geodesic} and $\gamma $ is a geodesic path of length $\leq\epsilon $ then $\gamma $ is the shortest path joining its endpoints.
%\end{Rem} THIS IS FALSE!!!!
We recall that a metric space is called {\it proper} if closed balls are compact.

\begin{Lem} \label{close} Let $X$ be a geodesic, proper, and $\epsilon $-locally uniquely geodesic metric space. 
Let $a_n,b_n\in X$ such that $a_n\to a, b_n\to b$ and $d(a_n,b_n)<\epsilon$ for all $n$.
If $\gamma _n:[0,1]\to X$ are shortest paths joining $a_n$ to $b_n$ then $\gamma _n$ converges uniformly
to the unique shortest path $\gamma $ joining $a,b$.

\end{Lem}
\begin{proof}
The geodesics the $\gamma_n$ are all Lipschitz with constant $\epsilon$, and so equicontinuous.  
Since $X$ is proper, by Arzela-Ascoli,  a subsequence of $(\gamma _n)$ converges uniformly to a path $\beta $ joining $a,b$.
Since $\length(\gamma_n) =d(a_n,b_n) \to d(a,b)$ and  $\length (\gamma _n)\to \length (\beta)$,   $\beta $ is the unique shortest path joining $a$ and $b$.
Since $X$ is locally uniquely geodesic $\beta =\gamma $. If $\gamma _n$ does not converge uniformly to $\beta $
then a subsequence of $\gamma _n$ converges uniformly to a shortest path different from $\beta $, contradicting uniquely geodesic.
\end{proof}
It follows that shortest paths of length less than $\epsilon$ vary continuously with their endpoints in an $\epsilon$-locally uniquely geodesic space.
\section{ The Birkhoff shortening process} 

We generalize below the Birkhoff shortening process in the context of $\epsilon $-locally uniquely geodesic metric space. We note that Bowditch \cite{Bo} has generalized this to CAT(1) spaces.

Let $X$ be an $\epsilon $-locally uniquely geodesic metric space and let $\gamma :[0,1]\to X$ be a continuous closed path. 
Let $k$ be an integer such that the set $\gamma [ t ,t +\frac 1k]$ (with $t+\frac 1 k$ taken mod 1 if needed) has diameter  less than $\frac \epsilon 2 $ for all $t \in [0,1]$ .

We define a process that will shorten this curve in two stages. In the first stage we consider all
integers $0\le i<k$ and define a homotopy $R_t$, where $t\in [0,\frac 1 k]$ by replacing the interval $\gamma ( [\frac {2i}{2k},\frac{2i}{2k}+t])$ with the shortest path with the same endpoints. For $t=\frac 1 k$ we obtain a path $\gamma _1$ consisting
of the unique shortest paths from $ \gamma(\frac {2i}{2k})$ to $\gamma(\frac {2i+2}{2k})$ for each $0\le i<k$. 

We parametrize each shortest path proportionally to arc length so that
$\gamma  _1 (\frac {2i}{2k})=\gamma (\frac{2i}{2k})$ for each integer  $0\le i<k$. Since $\gamma_1$ is piecewise constant speed,  $\gamma_1$ is Lipschitz (even though $\gamma$ may not have been Lipschitz at all) with constant equal to the maximal speed 

\begin{equation} \label{Eq:speed} k \max\limits_{0\le i<k}  d\left(\gamma\left(\frac {2i}{2k}\right), \gamma\left(\frac {2i+2}{2k}\right)\right). \end{equation}

 If follows immediately that if $\gamma$ was Lipschitz with constant $\mu$ then $\gamma_1$ is Lipschitz with constant $\mu$.   

For integers $0 \le i <k$, define the homotopy
$S_t$ ,  $t\in [0,\frac 1 k]$ by replacing the interval $\gamma _1 ( [\frac {2i+1}{2k},\frac{2i+1}{2k} +t])$ with the shortest path with the same endpoints (where we consider the numbers $\mod 1$ so $0\in [\frac{2k-1}{2k},\frac 1{2k}]$). So $S_{1/k}(\gamma  _1)$ is the path consisting of shortest paths from
$\gamma_1(\frac{2i+1}{2k})$ to $\gamma_1(\frac {2i+3}{2k})$, where the numbers are taken mod 1 and $0\le i <k$.  As before, if $\gamma_1$ was Lipschitz with constant $\mu$, so is $S_{1/k}(\gamma  _1)$.

$S_{1/k}(\gamma  _1)$ is the outcome of the Birkhoff shortening process. We denote by $D_t$ the homotopy from $\gamma $ to $S_{1/k}(\gamma  _1)$ and we set
 $D(\gamma )=S_{1/k}(\gamma  _1)$.  Notice that we have shown that if $\gamma$ was Lipschitz with constant $\mu$ so is $D(\gamma)$

 \section{A foliation of the sphere} 

We will need a standard foliation of the sphere $S^n$ by circles that we describe now ($n\geq 2$).
Let $S^n\subseteq \mathbb R^{n+1}$ be the standard sphere. If $e_1,...,e_{n+1}$ is the standard basis of $\mathbb R^{n+1}$  we denote by $P_{x_1,...,x_{n-1}}$ the affine plane perpendicular
to $V=span\,(e_1,...,e_{n-1})$ which intersects $V$ at $(x_1,....,x_{n-1},0,0)$. Each plane $P_{x_1,...,x_{n-1}}$ intersects $S^n$ along a circle or a point (or has empty intersection).

We would like to pick base points on the circles of the foliation in a continuous fashion.

We pick the base point of each circle  in the foliation to be the unique point with coordinates
satisfying $x_n=0,x_{n+1}\geq 0$.  We now parameterize the the intersection $S^n \cap P_{x_1,...,x_{n-1}}$ explicitly.
So assuming $S^n \cap P_{x_1,...,x_{n-1}} \neq \emptyset$, then let $r = \sqrt{ 1- \sum_{i=1}^{n-1} x_i^2}$, with $0 \le r \le 1$.
Now we define $\phi_{x_1,...,x_{n-1}}:[0,1] \to S^n$ by $$\phi_{x_1,...,x_{n-1}}(t) = (x_1,\dots, x_{n-1}, r\sin2\pi t, r\cos 2\pi t)$$
Thus $\phi_{x_1,...,x_{n-1}}(0) =\phi_{x_1,...,x_{n-1}}(1) = (x_1,\dots, x_{n-1}, 0, r)$ our chosen base point.
We note that the set of base points is equal to a closed half sphere
of dimension $n-1$, so it is homeomorphic to a disk $B^{n-1}$.

\section{Periodic geodesics in compact spaces}

\begin{Thm} Let $X$ be a compact locally uniquely geodesic metric space with $\pi _n(X)\ne 0$
for some $n\geq 1$. Then $X$ contains a periodic geodesic.
\end{Thm}
\proof
Let's say that $X$ is $\epsilon $-locally uniquely geodesic.

We treat first the $n=1$ case. Since $X$ is locally uniquely geodesic there is a non-contractible
closed path $\gamma :[0,1]\to X$ of finite length. We note that any closed curve
of length $\leq \epsilon $ is contractible. Indeed by lemma \ref{close} if we join a base point
of the curve to the other points we obtain a contraction.

It follows that $\length (\gamma )>\epsilon$. Let $c$ by the infimum of the lengths
of non-contractible closed paths in $X$. If $\gamma _n$ is a sequence of such closed paths
such that $\length (\gamma _n)\to c$ then since $X$ is compact by passing to a subsequence
we get that $\gamma _n\to \beta $ and by lemma \ref{close} $\beta$ is a periodic geodesic.

We assume now that $n\geq 2$. Let $f:S^n\to X$ be a non-contractible map.
We consider the foliation of $S^n$ by circles defined in the previus section and set
$f_{x_1,...,x_{n-1}}= f\circ\phi_{x_1,...,x_{n-1}}:[0,1]\to X$, the restriction of $f$ to one of these circles.
Since $S^n$ is compact and $f$ is continuous there is some $k\in \mathbb N$
such that the diameter of $$f_{x_1,...,x_{n-1}}([ t, t+ 1/k])$$ is bounded by $\epsilon /2$
for all $x_1,...,x_{n-1}$ with $x_1^2+\dots +x_{n-1}^2 \le 1$ and for all $t$ (taken mod 1).

Abusing notation slighlty we write below sometimes $\bar f$ instead of $f_{x_1,...,x_{n-1}}$
to simplify notation.
We apply the Birkhoff shortening process to each $\bar f:[0,1]\to X$ and we obtain
a homotopic curve of finite length $D\bar f$. By lemma \ref{close}, $D$ applied to
each circle induces a  map from $S^n$ to $X$ homotopic to $f$ which we denote by$Df$.

We define inductively $D^kf=D\circ D^{k-1}f$ and define $D^k\bar f$ similarly.

Let 

$$c_k= \max \{\length (D^kf_{x_1,...,x_{n-1}}): x_1^2+\dots +x_{n-1}^2 \le 1\}.$$

Clearly $c_k$ is decreasing. Let $c=\lim c_k$.

We claim that $c\geq \epsilon $. Indeed if not $c_k<\epsilon $ and then we can contract all $D^k\bar f$ to their base points by a homotopy by lemma \ref{close}. This homotopy is continuous on $S^n$
by lemma \ref{close}.

Since the set of basepoints of the circles
of the foliation is a topological disc this shows that $D^nf$ is homotopically trivial, a contradiction since $D^nf$ is homotopic to $f$.

Let $\bar f _k:[0,1]\to X$ be curve realizing $c_k$, so  $ \bar f_k =D^kf_{x_1,...,x_{n-1}}$  for some fixed $x_1^2+\dots +x_{n-1}^2 \le 1$.
Let $g_{k-1} = D^{k-1}f_{x_1,...,x_{n-1}}$, so $Dg_{k-1} = \bar f_k$.
 We have
$c_k\leq \length (g_{k-1})\leq c_{k-1}.$

We must now show that the sequence $(g_k)$ is equicontinuous.  We do this by showing that there is an $\mu>0$ such that $g_k$ is $\mu$-Lipschitz for all $k>0$.  We already know that if $\beta:[0,1]\to X$ is Lipschitz with constant $\mu$ then so is $D(\beta)$.   Fix $x_1, \dots, x_{n-1}$ such that $x_1^2 +\dots +x_{n-1}^2 \le 1$. By equation \ref{Eq:speed} in the Birkhoff shortening process, $Df_{x_1, \dots, x_{n-1}}$ is Lipschitz with constant 
$$ k \max\limits_{0\le i<k}  d\left(f_{x_1, \dots, x_{n-1}}\left(\frac {2i}{2k}\right), f_{x_1, \dots, x_{n-1}}\left(\frac {2i+2}{2k}\right)\right).$$
By our hypothesis on $k$ $$d\left(f_{x_1, \dots, x_{n-1}}\left(\frac {2i}{2k}\right), f_{x_1, \dots, x_{n-1}}\left(\frac {2i+2}{2k}\right)\right)< \frac \epsilon 2$$
It follows that $Df_{x_1, \dots, x_{n-1}}$ is Lipschitz with constant $\frac {k\epsilon}2$.  This is independent of the choice of $x_1, \dots x_{n-1}$, and since $D$ of a Lipschitz $\frac {k\epsilon}2$ function is a Lipschitz $\frac {k\epsilon}2$ function, it follows that $g_k$ is Lipschitz with constant $\frac {k\epsilon}2$ for all $k>0$.

Thus by Arzela-Ascoli, passing to a subsequence we may assume that $g_k$ converges uniformly to a curve $g:[0,1]\to X$.
Clearly $\length (g)=c$ and by lemma \ref{close} $g$ is a periodic geodesic.

\qed

\begin{Cor} 1. Let $X$ be a compact $CAT(\kappa )$ manifold.
Then $X$ contains a periodic geodesic.

2.  Let $X$ be a compact $CAT(\kappa )$ polyhedron such that $\pi _n(X)\ne 0$
for some $n\geq 1$. Then $X$ contains a periodic geodesic.

3. Let $X$ be a finite dimensional non-contractible compact locally uniquely geodesic space. Then $X$ contains a periodic geodesic.
\end{Cor}

\proof
For part 1, taking a double cover if need be, we may assume that $X$ is orientable and connected, and so $H_n(X) =\Z$ where $n = \dim X$.  By the Hurewicz Theorem, $\pi_k(X)$ is non trivial for some $0<k\le n$, and the result follows from the Theorem.

Part  2 follows directly from the Theorem.

For part 3, let's assume that $X$ is $\epsilon$-locally uniquely geodesic.  Then every metric ball of radius at most $\epsilon$ is contractible via the ``straight line" contraction.  Thus $X$ is locally contractible.   By \cite[Corollary V 10.4]{BOR}, $X$ is an ANR (absolute neighborhood retract).  
The metric space $X$ is separable since it is compact.  It now follows from \cite[Theorem 1]{MIL} that $X$ is homotopy equivalent to a countable CW complex.  
(The infinite dimensional case  would also work if metric balls of radius less than $\epsilon$ were convex via \cite[Lemma 4]{MIL}.)
Hence by Whitehead's Theorem $\pi _n(X)\ne 0$ for some $n$,  and the Theorem applies.

\qed

\begin{Rem} It is worth noting that the class of locally uniquely geodesic spaces is strictly larger than the class
of $CAT(\kappa )$ spaces. To see this pick a sequence of small spherical caps $C_n$  in spheres 
of radius $1/2^n$ drill holes on a plane (or a sphere) converging to a point so that the distance
of successive holes is $1/n$ and glue the caps $C_n$ to these holes. The resulting space is locally uniquely geodesic but not $CAT(\kappa )$.
\end{Rem}

\end{document}
\bye